\documentclass[12pt,a4paper]{article}
\usepackage{amsfonts}
\usepackage{amssymb}
\usepackage{mathrsfs}
\usepackage{amsmath}
\usepackage{amsthm}
\usepackage{enumerate}
\usepackage{cite}
\allowdisplaybreaks
\newtheorem{lemma}{Lemma}[section]
\newtheorem{theorem}[lemma]{Theorem}
\newtheorem{remark}[lemma]{Remark}
\newtheorem{proposition}[lemma]{Proposition}
\newtheorem{definition}[lemma]{Definition}
\newtheorem{corollary}[lemma]{Corollary}
\newtheorem{example}[lemma]{Example}

\begin{document}

\title{\textbf{Some characterizations of nilpotent $n$-Lie superalgebras}
\author{ Baoling Guan$^{1,2},$  Liangyun Chen$^{1},$    Ma Yao$^{1}$
 \date{{\small {$^1$ School of Mathematics and Statistics, Northeast Normal
 University,\\
Changchun 130024, China}\\{\small {$^2$ College of Sciences, Qiqihar
University, Qiqihar 161006, China}}}}}}

\date{ }
\maketitle
\begin{quotation}
\small\noindent \textbf{Abstract}: The paper studies nilpotent
$n$-Lie superalgebras. More specifically speaking, we first prove
Engel's theorem for $n$-Lie superalgebras.  Second,
we research some properties of nilpotent $n$-Lie superalgebras,
Finally, we give several sufficient conditions that an $n$-Lie
superalgebra is nilpotent.

\noindent{\textbf{Keywords}}:  Nilpotent $n$-Lie superalgebras;
Engel's theorem; $S^{\ast}$ algebra; Frattini subalgebra

\small\noindent \textbf{Mathematics Subject Classification 2000}:
17B45, 17B50
\renewcommand{\thefootnote}{\fnsymbol{footnote}}
 \footnote[0]{Corresponding
author(L. Chen): chenly640@nenu.edu.cn.}
 \footnote[0]{Supported by NNSF
of China (No.11171055),  Natural Science Foundation of Jilin
province (No. 201115006), Scientific Research Foundation for
Returned Scholars
    Ministry of Education of China and the Fundamental Research Funds for the Central Universities(No. 11SSXT146).}
\end{quotation}
\setcounter{section}{0}

\section{Introduction}
The nilpotent theories of many algebras attract more and more
attentions.
 For example: In  {\rm
  \cite{ccglo,cana,mpw}}, the authors study  nilpotent Leibniz $n$-algebras, nilpotent Lie and Leibniz algebras, nilpotent $n$-Lie algebras,
  respectively;
  D. W. Barnes discusses Engel subalgebras of Leibniz algebras in {\rm
  \cite{dowb}, and so on.
 In 2010, the concept of $n$-Lie
superalgebras was introduced by Cantarini, N. and Kac V. G. in {\rm
  \cite{ck}}.  $n$-Lie superalgebras are
generalization of $n$-Lie algebras. As the structural properties of
$n$-Lie superalgebras mostly remains unexplored and motivated by the
investigation on Engel's theorem and nilpotentcy of $n$-Lie
algebras{\rm
  \cite{dwb,cyc,ce,ksh,mpw}
and Lebniz $n$-algebras {\rm
  \cite{aaot,ccglo,ckl,glot}}, it is natural to ask about the extension of these properties to the $n$-Lie superalgebras category. As is well known, for $n$-Lie algebras and Leibniz $n$-algebras,
Engel's theorem and nilpotentcy play a predominant role in Lie
theory. Analogously,
 Engel's Theorem and nilpotentcy for $n$-Lie superalgebras will also play an important
role in Lie theory.

The goal of the present paper is to study Engel's theorem and
nilpotentcy for $n$-Lie superalgebras.
 We first prove
Engel's theorem for $n$-Lie superalgebras, which will generalize
Engel's theorems for $n$-Lie algebras and Lie superalgebras, then we
research some properties of nilpotent $n$-Lie superalgebras,
moreover, we give several sufficient conditions that an $n$-Lie
superalgebra is nilpotent.

\begin{definition}{\rm \cite{ck}}\label{d1.1}\,  An $n$-Lie superalgebra is an anti-commutative $n$-superalgebra $A$ of parity $\alpha,$
such that all endomorphisms $D(a_{1},\cdots,a_{n-1})$ of
$A(a_{1},\cdots,a_{n-1}\in A),$ defined by
           $$D(a_{1},\cdots,a_{n-1})(a_{n})=[a_{1},\cdots,a_{n-1},a_{n}],$$
are derivations of $A,$ i.e., the following Filippov-Jacobi identity
holds:
\begin{eqnarray*}&& [a_{1},\cdots,a_{n-1},[b_{1},\cdots,b_{n}]] =(-1)^{\alpha(p(a_{1})+\cdots+p(a_{n-1}))}([[a_{1},\cdots,a_{n-1},b_{1}],b_{2},\cdots,b_{n}]\\
&&\quad+(-1)^{p(b_{1})(p(a_{1})+\cdots+p(a_{n-1}))}[b_{1},[a_{1},\cdots,a_{n-1},b_{2}],b_{3},\cdots,b_{n}]+\cdots\\
&&\quad+(-1)^{(p(b_{1})+\cdots+p(b_{n-1}))(p(a_{1})+\cdots+p(a_{n-1}))}[b_{1},\cdots,b_{n-1},[a_{1},\cdots,a_{n-1},b_{n}]]).
\end{eqnarray*}
\end{definition}

From the above definition, we may see that
$p([a_{1},\cdots,a_{n}])=\alpha+\sum\limits_{i=1}^{n}p(a_{i})$ and $
[a_{1},\cdots,a_{i},a_{i+1},\cdots,a_{n}]=-(-1)^{p(a_{i})p(a_{i+1})}
[a_{1},\cdots,a_{i+1},a_{i},\cdots,a_{n}], \forall a_{i}\in A(1\leq
i\leq n),$ where $p([a_{1},\cdots,a_{n}])$ and $p(a_{i})$ denote the
degrees of $[a_{1},\cdots,a_{n}]$ and $a_{i}$, respectively.
Moreover, since $n$-Lie superalgebra $A$ is related to $\alpha,$ it
is also denoted by $(A,\alpha).$

Analogous to $n$-Lie algebras(\rm\cite{ksh}), we have the following
definition:

\begin{definition}  Let $A=A_{\bar{0}}\oplus A_{\bar{1}}$ be an $n$-Lie superalgebra and $I$ be a subspace of $A.$

$\mathrm{(i)}$ $I$ is called a vector superspace, if
$I=I_{\bar{0}}\oplus I_{\bar{1}},$ where $I_{\bar{0}}=I\cap
A_{\bar{0}},$ $I_{\bar{1}}=I\cap A_{\bar{1}};$

$\mathrm{(ii)}$ A vector superspace $I\subseteq A$ is called a
subalgebra, if $[I,I,\cdots,I,I]\subseteq I;$

 $\mathrm{(iii)}$ A vector superspace $I\subseteq A$ is called an ideal($I\lhd A$), if $[A,A,\cdots,A,I]\subseteq I;$

 $\mathrm{(iv)}$ A vector superspace $I\subseteq A$ is called a weak ideal, if $[A,I,\cdots,I,I]\subseteq I;$

 $\mathrm{(v)}$ An ideal $I$ is called abelian, if $[A,A,\cdots,A,I,I]=0;$

 $\mathrm{(vi)}$ An ideal $I$ of an algebra $A$ is called nilpotent, if $I^{v}=0$ for some $v\geq 0,$ where $I^{1}=I, I^{s+1}=[A,\cdots,A,I,I^{s}].$

 In sequel, Let $\mathbb{F}$ be an arbitrary field and $A$ be a finite dimensional $n$-Lie superalgebra over a field
 $\mathbb{F}.$
\end{definition}
\section{Engel's theorem of $n$-Lie superalgebras}
\begin{definition}\label{d1.3} Let $A=A_{\bar{0}}\oplus A_{\bar{1}}$ be an $n$-Lie superalgebra over a field $\mathbb{F}.$
A vector superspace $V$ over $\mathbb{F}$ is called an $A$-module if
there is defined on the direct sum of vector space $V\oplus A=B$ the
structure of an $n$-Lie superalgebra such that $A$ is a subalgebra
of $B$ and $V$ is an abelian ideal of $B.$
\end{definition}

\begin{definition} \label{d1.4} Let $A=A_{\bar{0}}\oplus A_{\bar{1}}$ be a vector superspace over a field $\mathbb{F}$ and $(A,\alpha)$
 be an $n$-Lie superalgebra over $\mathbb{F}.$ We define a
 multilinear mapping
$\rho: A^{\times(n-1)}=\overbrace{A\times A\times\cdots\times
A}^{n-1}\rightarrow\mathrm{End}V,
(x_{1},x_{2},\cdots,x_{n-1})\mapsto\rho(x_{1},\cdots,x_{n-1}).$ Then
$\rho$ is called a representation and $V$ is called an $A$-module,
if the following relations are satisfied:
\begin{equation} \label{e1.1} \mathrm{(1)}\ \rho(a_{1},\cdots,a_{i},a_{i+1},\cdots,a_{n-1})=-(-1)^{p(a_{i})p(a_{i+1})}
 \rho(a_{1},\cdots,a_{i+1},a_{i},\cdots,a_{n-1}),a_{i}\in A.\quad\quad
   \end{equation}
\begin{equation} \label{e1.2}\mathrm{(2)}\rho(b)\rho(a)=(-1)^{p(a)(p(b)+\alpha)}\rho(a)\rho(b)+\sum_{i=1}^{n-1}(-1)^{p(b)(\sum\limits_{j=1}^{i-1}p(a_{j})+\alpha)}
        \rho(a_{1},\cdots,D(b)(a_{i}),\cdots,a_{n-1}),
   \end{equation}
 where  $a=(a_{1},\cdots,a_{n-1}), b=(b_{1},\cdots,b_{n-1}),a_{i}, b_{i}\in A.$
\begin{equation} \label{e1.33}\mathrm{(3)}  \rho(a_{1},\cdots,a_{n-2},[b_{1},\cdots,b_{n}])(c)=
       \sum_{i=1}^{n} \lambda_{i}\rho(b_{1},\cdots,\hat{b_{i}},\cdots,b_{n})\rho(a_{1},\cdots,a_{n-2},b_{i})(c),\quad\quad\quad
\end{equation}
where $\lambda_{i}=(-1)^{n-i}(-1)^{p(a)(\sum\limits_{j=1 j\neq
i}^{n}p(b_{j}))+(p(b_{i})+\alpha)(\sum\limits_{j=i+1}^{n}p(b_{j}))}(-1)^{\alpha(p(a_{1})+p(a_{2})+\cdots+p(a_{n-2}))},$
\\ $p(a)=\sum\limits_{i=1}^{n-2}p(a_{i}),$ $\hat{b_{i}}$ denotes $b_{i}$ is
omitted, and $a_{i}, b_{i}, c\in A.$\\
 $\mathrm{(4)}
\rho(a)(V_{\theta})\subseteq V_{\theta+\beta},$ where
$a=(a_{1},\cdots,a_{n-1}),\theta\in\mathbb{Z}_{2},
\beta=p(a)=\sum\limits_{i=1}^{n-1}p(a_{i}), a_{i}\in A.$
\end{definition}

\begin{remark}\label{r1.1} Definition $\ref{d1.4}$ is equivalent to Definition $\ref{d1.3}.$
Definition $\ref{d1.4}$ can conclude Definition $\ref{d1.3}.$ In
fact, let $\rho$ be a representation of $A$ and $V$ be an
$A$-module. Then $\rho$ is a linear transformation on $V.$  We can
define on the direct sum of linear spaces $V\oplus A$ a
skew-super-symmetric $n$-ary operator:
$$[x_{1},\cdots,x_{n-2},v_{1},v_{2}]:=0, [x_{1},\cdots,x_{n-1},v]=\rho(x_{1},\cdots,x_{n-1})(v)\in V,$$ where $ x_{1},\cdots,x_{n-2}\in A,v_{1},v_{2},v \in V.$
For $ x_{1},\cdots,x_{n-1}, y_{1},\cdots,y_{n-1}\in A, v\in V,$ by
$(\ref{e1.1})$, we have
\begin{eqnarray*}&& [x_{1},\cdots,x_{n-1},[y_{1},\cdots,y_{n-1},v]]=\rho(x)\rho(y)(v)= (-1)^{p(y)(p(x)+\alpha)}\rho(y)\rho(x)(v)\\
&&\quad+\sum\limits_{i=1}^{n-1}(-1)^{p(x)(\sum\limits_{j=1}^{i-1}p(y_{j})+\alpha)}\rho(y_{1},\cdots,D(x)(y_{i}),\cdots,y_{n-1})(v)\\
&&= (-1)^{p(x)(p(y)+\alpha)}[y_{1},\cdots,y_{n-1},[x_{1},\cdots,x_{n-1},v]]\\
&&\quad+(-1)^{p(x)\alpha}[[x_{1},\cdots,x_{n-1},y_{1}], y_{2},\cdots,y_{n-1},v]\\
&&\quad+(-1)^{p(x)(p(y_{1})+\alpha)}[y_{1},[x_{1},\cdots,x_{n-1},y_{2}], y_{3},\cdots,y_{n-1},v]\\
&&\quad+(-1)^{p(x)(p(y_{1})+p(y_{2})+\alpha)}[y_{1},y_{2},[x_{1},\cdots,x_{n-1},y_{3}], y_{4},\cdots,y_{n-1},v]\\
&&\quad+\cdots+(-1)^{p(x)(p(y_{1})+\cdots+p(y_{n-2})+\alpha)}[y_{1},\cdots,y_{n-2},[x_{1},\cdots,x_{n-1},y_{n-1}],v]\\
&&= (-1)^{p(x)\alpha}([[x_{1},\cdots x_{n-1},y_{1}],y_{2},\cdots,y_{n-1},v]\\
&&\quad+(-1)^{p(x)p(y_{1})}[y_{1},[x_{1},\cdots,x_{n-1},y_{2}], y_{3},\cdots,y_{n-1},v]\\
&&\quad+(-1)^{p(x)(p(y_{1})+p(y_{2}))}[y_{1},y_{2},[x_{1},\cdots,x_{n-1},y_{3}], y_{4},\cdots,y_{n-1},v]\\
&&\quad+\cdots+(-1)^{p(x)(p(y_{1})+\cdots+p(y_{n-2}))}[y_{1},\cdots,y_{n-2},[x_{1},\cdots,x_{n-1},y_{n-1}],v]\\
&&\quad+(-1)^{p(x)p(y)}[y_{1},\cdots,y_{n-1},[x_{1},\cdots,x_{n-1},v]]),
\end{eqnarray*}
where $p(x)=\sum\limits_{i=1}^{n-1}p(x_{i}),
p(y)=\sum\limits_{i=1}^{n-1}p(y_{i}),$ that is, the above formula
satisfies Filippov-Jacobi identity. Hence $V\oplus A$ is an $n$-Lie
superalgebra on the above operator such that $A$ is a subalgebra of
$V\oplus A$ and $V$ is an abelian ideal of $V\oplus A$.

Definition $\ref{d1.3}$ can also conclude Definition $\ref{d1.4}.$
In fact, for any $a_{1},\cdots,a_{n-1}\in A,$ there is a
corresponding linear transformation $\rho(a_{1},\cdots, a_{n-1})$ of
$V,$ where $\rho(a_{1},\cdots, a_{n-1})(v)=[a_{1},\cdots, a_{n-1},
v].$ Then the operators $\rho(a)$ satisfy the formulas
$(\ref{e1.1}),(\ref{e1.2})$ and $(\ref{e1.33}).$
 In fact,  it is clear that $(\ref{e1.1})$ holds.
\begin{eqnarray*}&& \rho(b)\rho(a)(c)= [b_{1},\cdots,b_{n-1},[a_{1},\cdots,a_{n-1},c]]\\
&&= (-1)^{\alpha p(b)}\{[[b_{1},\cdots,b_{n-1},a_{1}],a_{2},\cdots,a_{n-1},c]\\
&&\quad+(-1)^{p(b)p(a_{1})}[a_{1},[b_{1},\cdots,b_{n-1},a_{2}], a_{3},\cdots,a_{n-1},c]\\
&&\quad+(-1)^{p(b)(p(a_{1})+p(a_{2}))}[a_{1},a_{2},[b_{1},\cdots,b_{n-1},a_{3}], a_{4},\cdots,a_{n-1},c]\\
&&\quad+\cdots+(-1)^{p(b)(p(a_{1})+\cdots+p(a_{n-2}))}[a_{1},\cdots,a_{n-2},[b_{1},\cdots,b_{n-1},a_{n-1}],c]\\
&&\quad+(-1)^{p(b)(p(a_{1})+\cdots+p(a_{n-1}))}[a_{1},\cdots,a_{n-1},[b_{1},\cdots,b_{n-1},c]]\}\\
&&= (-1)^{p(b)\alpha}\rho(D(b)(a_{1}),a_{2},\cdots, a_{n-1})(c)\\
&&\quad+(-1)^{p(b)(p(a_{1})+\alpha)}\rho(a_{1},D(b)(a_{2}),a_{3},\cdots, a_{n-1})(c)\\
&&\quad+(-1)^{p(b)(p(a_{1})+p(a_{2})+\alpha)}\rho(a_{1},a_{2},D(b)(a_{3}),a_{4},\cdots, a_{n-1})(c)\\
&&\quad+\cdots+(-1)^{p(b)(p(a_{1})+\cdots+p(a_{n-2})+\alpha)}\rho(a_{1},\cdots, a_{n-2},D(b)(a_{n-1}))(c)\\
&&\quad+(-1)^{p(b)(p(a)+\alpha)}\rho(a)\rho(b)(c)\\
&&=(-1)^{p(b)(p(a)+\alpha)}\rho(a)\rho(b)(c)+\sum_{i=1}^{n-1}(-1)^{p(b)(\sum\limits_{j=1}^{i-1}p(a_{j})+\alpha)}
        \rho(a_{1},\cdots,D(b)(a_{i}),\cdots,a_{n-1})(c),
\end{eqnarray*}
where $D(b)=D(b_{1},\cdots,b_{n-1}),$
 that is, $(\ref{e1.2})$ holds.
\begin{eqnarray*}&& (-1)^{\alpha(p(a_{1})+p(a_{2})+\cdots+p(a_{n-2})+p(c))}\rho(a_{1},\cdots,a_{n-2},[b_{1},\cdots,b_{n}])(c)\\
&&= (-1)^{\alpha(p(a_{1})+p(a_{2})+\cdots+p(a_{n-2})+p(c))}[a_{1},\cdots,a_{n-2},[b_{1},\cdots,b_{n}],c]\\
&&= (-1)^{\alpha(p(a_{1})+p(a_{2})+\cdots+p(a_{n-2})+p(c))}(-(-1)^{p(c)(p(b_{1})+\cdots+p(b_{n})+\alpha)}[a_{1},\cdots,a_{n-2},c,[b_{1},\cdots,b_{n}]])\\
&&=-(-1)^{p(c)(p(b_{1})+\cdots+p(b_{n})+\alpha)}\{[[a_{1},\cdots,a_{n-2},c,b_{1}], b_{2},\cdots,b_{n}]\\
&&\quad+(-1)^{p(b_{1})(p(a_{1})+\cdots+p(a_{n-2})+p(c))}[b_{1},[a_{1},\cdots,a_{n-2},c,b_{2}], b_{3},\cdots,b_{n}]\\
&&\quad+\cdots+(-1)^{(p(b_{1})+\cdots+p(b_{n-1}))(p(a_{1})+\cdots+p(a_{n-2})+p(c))}
[b_{1},\cdots,b_{n-1},[a_{1},\cdots,a_{n-2},c,b_{n}]]\}\\
&&=-(-1)^{p(c)(p(b_{1})+\cdots+p(b_{n})+\alpha)}\{(-1)^{n+(p(a_{1})+\cdots+p(a_{n-2})+p(c)+p(b_{1})+\alpha)(p(b_{2})+\cdots+p(b_{n}))+p(b_{1})p(c)}\\
&&\quad\quad\quad.[b_{2},\cdots,b_{n},[a_{1},\cdots, a_{n-2},b_{1},c]]\\
&&\quad+(-1)^{n-1+p(b_{1})(p(a_{1})+\cdots+p(a_{n-2})+p(c))+(p(a_{1})+\cdots+p(a_{n-2})+p(c)+p(b_{2})+\alpha)(p(b_{3})+\cdots+p(b_{n}))+p(b_{2})p(c)}\\
&&\quad\quad\quad.[b_{1},b_{3},\cdots,b_{n},[a_{1},\cdots, a_{n-2},b_{2},c]]\\
&&\quad+\cdots+(-1)^{1+(p(b_{1})+\cdots+p(b_{n-1}))(p(a_{1})+\cdots+p(a_{n-2})+p(c))+p(b_{n})p(c)}\\
&&\quad\quad\quad.[b_{1},\cdots,b_{n-1},[a_{1},\cdots, a_{n-2},b_{n},c]]\}\\
&&=(-1)^{n+1+\alpha p(c)+(p(a_{1})+\cdots+p(a_{n-2}))(p(b_{2})+\cdots+p(b_{n}))+(p(b_{1})+\alpha)(p(b_{2})+\cdots+p(b_{n}))}\\
&&\quad\quad\quad.\rho(b_{2},\cdots,b_{n})\rho(a_{1},\cdots, a_{n-2},b_{1})(c)\\
&&\quad+(-1)^{n+\alpha p(c)+(p(a_{1})+\cdots+p(a_{n-2}))(p(b_{1})+p(b_{3})+\cdots+p(b_{n}))+(p(b_{2})+\alpha)(p(b_{3})+\cdots+p(b_{n}))}\\
&&\quad\quad\quad
.\rho(b_{1},b_{3},\cdots,b_{n})\rho(a_{1},\cdots, a_{n-2},b_{2})(c)\\
&&\quad+(-1)^{n-1+\alpha
p(c)+(p(a_{1})+\cdots+p(a_{n-2}))(p(b_{1})+p(b_{2})+p(b_{4})+\cdots+p(b_{n}))+(p(b_{3})+\alpha)(p(b_{4})+\cdots+p(b_{n}))}
\\&&\quad\quad\quad .\rho(b_{1},b_{2},b_{4},\cdots,b_{n})\rho(a_{1},\cdots, a_{n-2},b_{3})(c)\\
&&\quad+\cdots+(-1)^{2+\alpha
p(c)+(p(a_{1})+\cdots+p(a_{n-2}))(p(b_{1})+\cdots+p(b_{n-1}))}
.\rho(b_{1},\cdots,b_{n-1})\rho(a_{1},\cdots, a_{n-2},b_{n})(c)\\
&&=\sum_{i=1}^{n} (-1)^{n-i+\alpha
p(c)}(-1)^{p(\sum\limits_{j=1}^{n-2}p(a_{j}))(\sum\limits_{j=1 j\neq
i}^{n}p(b_{j}))+(p(b_{i})+\alpha)(\sum\limits_{j=i+1}^{n}p(b_{j}))}\\
&&\quad\quad\quad.\rho(b_{1},\cdots,\hat{b_{i}},\cdots,b_{n})\rho(a_{1},\cdots,a_{n-2},b_{i})(c),
\end{eqnarray*}
that is, $(\ref{e1.33})$ holds.
\end{remark}

A special case of the representation is the regular representation
$a \mapsto D(a),$ where $D(a)=D(a_{1},\cdots,a_{n-1}),$
$D(a)(a_{n})=[a_{1},\cdots,a_{n-1},a_{n}], a_{i}\in A.$ The subspace
$\mathrm{ker}\rho=\{x\in A|\rho(A,\cdots,A,x)=0\}$ is called the
kernel of the representation $\rho.$ It follows from $(\ref{e1.1})$
that $\mathrm{ker}\rho\lhd A.$ If $\mathrm{ker}\rho=0,$ then the
representation $\rho$ is called faithful. A subset $S\subseteq A$
will be called homogeneous multiplicatively closed(h.m.c.), if for
any $x,x_{1},\cdots,x_{n}\in S, \lambda\in \mathbb{F},$ we have
$\lambda x\in S, [x_{1},\cdots,x_{n}]\in S.$ We denote the linear
span of a h.m.c. set $S$ by $F(S),$ it is clear that $F(S)$ is equal
to the subalgebra generated by the set $S.$\

\begin{theorem} \label{t1.2} Suppose that $\rho$ is a representation on an $n$-Lie superalgebra $A$ in a finite-dimensional space $V,$  $S$ is a h.m.c. subset of $A$ and the operators $\rho(a_{1},\cdots,a_{n-1})$ are nilpotent for any $a_{1},\cdots,a_{n-1}\in S.$ Then the algebra $S_{\rho}^{\ast}$ generated by these operators
is nilpotent. In addition, If the representation $\rho$ is faithful,
the algebra $F(S)$ is also nilpotent and acts nilpotently on $A.$
\end{theorem}
\begin{proof}
Passing to the quotient algebra $A/\mathrm{ker}\rho,$ we may assume
with no loss of generality that $\rho$ is faithful. To any subset
$X\subseteq S,$ we associate the subalgebra $X_{\rho}^{\ast}\leq
A_{\rho}^{\ast}$ generated by these operators
$\rho(a_{1},\cdots,a_{n-1}), a_{i}\in X.$ Suppose $X$ is a maximal
h.m.c. subset of $S$ for which the corresponding algebra
$X_{\rho}^{\ast}$ is nilpotent. Our aim is to prove that $X=S.$

Suppose $(X_{\rho}^{\ast})^{s}=0.$ Put $C=F(X),$ $C_{0}=A,$
$C_{i+1}=[C,\cdots,C,C_{i}]$ for $i\geq 0.$ We introduce an
abbreviated notation for certain subspaces of $A_{\rho}^{\ast}:$
$$\rho(A,\cdots,A,C_{i})=\rho(A,C_{i}),\rho(C,\cdots,C,A)=\rho(C,A),\rho(C,\cdots,C)=\rho(C),$$ etc.
We will show by induction on $k$ that for any $k\geq 0,$
 we have
\begin{equation} \label{e1.3}\rho(C,C_{k})\subseteq
\sum_{i=0}^{k}\rho^{i}(C)\rho(C,A)\rho^{k-i}(C).
\end{equation}
In fact, it follows from $(\ref{e1.2})$ that
$$\rho(C,C_{k+1})=\rho(C,[C,\cdots,C,C_{k}])\subseteq \rho(C,C_{k})\rho(C)+\rho(C)\rho(C,C_{k}).$$
This enables us to complete the inductive passage from $k$ to $k+1$ in relation $(\ref{e1.3}),$ which is trivial for $k=0.$
It follows from $(\ref{e1.2})$ that
\begin{equation} \label{e1.4}\rho(A,C_{k+1})=\rho(A,[C,\cdots,C,C_{k}])\subseteq
\rho(C,C_{k})\rho(A,C)+\rho(C)\rho(A,C_{k}).
\end{equation}
Again using induction on $k$ and $(\ref{e1.3}),$ we see that for $k\geq 1$
$$\rho(A,C_{k})\subseteq \rho^{k}(C)\rho(A)+\sum_{i+j=k-1}\rho^{i}(C)\rho(C,A)\rho^{j}(C)\rho(A,C). $$
Since $\rho^{s}(C)=0,$ for $k\geq 2s$ we obtain $\rho(A,C_{k})=0,$ i.e., $C_{k}\subseteq \mathrm{ker}\rho,$ hence $C_{k}=0.$
This means that $C$ acts nilpotently on $A$ by left multiplications, in particular, the algebra $C$ is itself nilpotent.

If $S\neq X,$ it follows easily from the preceding that $S\backslash X$ contains an element $b$ such that
\begin{equation} \label{e1.5}[X,\cdots,X,b]\subseteq X.
\end{equation}
Then $Y=\mathbb{F}b\cup X$ is  a h.m.c. subset of $S$ strictly
containing $X.$ We will show that the algebra $Y_{\rho}^{\ast}$ is
nilpotent, which is contrary to the maximality of $X.$ Any element
of $\rho(Y)$ lies either in  $\rho(X)$ or in $\rho(X,b).$ Suppose
$U\in \rho(Y)^{m}, m>0.$ If in the word $U$ the operators in
$\rho(X)$ occur at least $s$ times, in view of $(\ref{e1.1})$ and
$(\ref{e1.5}),$ then $U$ can be transformed into a sum of words in
which the operators in $\rho(X)$ appear consecutively and the number
of them is at least $s,$ therefore $U=0.$

On the other hand, if in $U$ the operators in $\rho(X)$ occur $l\leq
s-1$ times, then $U$ has the form $U_{1}\rho_{1}U_{2}\rho_{2}\cdots
U_{l}\rho_{l}U_{l+1},$ where $\rho_{i}\in \rho(X),$ $U_{i}$ are
products of elements $\rho(X,b),$ and some of the words $U_{i}$ can
be empty.

Let us view $A$ as an $(n-1)$-Lie superalgebra $A_{b}$ with
operation
$$[a_{1},\cdots,a_{n-1}]_{b}=[a_{1},\cdots,a_{n-1},b]$$ and $V$ as
an $A_{b}$-module on which there acts the representation
$\tilde{\rho}$ of the algebra $A_{b}:$
 $\tilde{\rho}(a_{1},\cdots,a_{n-2})=\rho(a_{1},\cdots,a_{n-2},b).$ It follows from $(\ref{e1.5})$ that $X$ is
 a h.m.c. set in $A_{b}.$ Since the operators in
 $\tilde{\rho}(X)=\rho(X,b)$ are nilpotent, the induction assumption with respect to $n$ is applicable to the
triple $(A_{b},X,\tilde{\rho})$ and the algebra
$X_{\tilde{\rho}}^{\ast}$ is nilpotent,
 say of index $t.$ When $n=2,$ since the algebra $X_{\tilde{\rho}}^{\ast}$ is generated by the nilpotent operator $\rho(b),$  $X_{\tilde{\rho}}^{\ast}$ is nilpotent,
  which provides the basis for the induction.

Therefore $U_{i}=0,i=1,\cdots,l+1,$ if the $\rho$-length of $U_{i}=0$ is greater than or equal to $t.$ Consequently, when $m\geq st$ all words  $U\in \rho(Y)^{m}$
are zero, i.e., $(Y_{\rho}^{\ast})^{st}=0$ as required. This contradiction shows that $X=S.$ The second assertion of the theorem has already been proved,
since $C=F(X)=F(S).$
\end{proof}

\begin{theorem}(Engel's Theorem) \label{t1.1} Suppose $A$ is a finite-dimensional $n$-Lie superalgebra in which
all left multiplication operators $D(a)$ are nilpotent, where
$D(a)=D(a_{1},\cdots,a_{n-1}),$ $ a_{i}\in A(1\leq i\leq n-1).$
Then $A$ is nilpotent.
\end{theorem}
\begin{proof}
Let $\rho$ be the regular representation and $A=V=S.$ By Theorem
$\ref{t1.2},$ we may obtain $A$ is nilpotent.
\end{proof}

\section{Nilpotency of $n$-Lie superalgebras}

\begin{definition} \label{d2.2}  The Frattini subalgebra, $F(A)$, of $A$ is the intersection of all maximal
subalgebra of $A.$  The maximal ideal of $A$ contained in  $F(A)$ is
denoted by $\phi(A).$
\end{definition}
The following proposition contain analogous results to the
corresponding ones for $n$-Lie algebras, their proof is similar to
$n$-Lie algebras (see {\rm
  \cite{bcm}}, Proposition 2.1).
\begin{proposition} \label{p2.3} Let $A$ be an  $n$-Lie superalgebra over $\mathbb{F}.$ Then the following statements hold:

$\mathrm{(1)}$ If $B$ is a subalgebra of $A$ such that $B+F(A)=A,$
then $B=A;$

$\mathrm{(2)}$ If $B$ is a subalgebra of $A$ such that
$B+\phi(A)=A,$ then $B=A.$
\end{proposition}
\begin{lemma} \label{l2.91} Let $A$ be an  $n$-Lie superalgebra over $\mathbb{F}.$ Then
$F(A)\subseteq A^{2};$ in particular, if $A$ is abelian, then
$F(A)=0.$
\end{lemma}
\begin{proof}
If $A=A^{2}=[A,\cdots,A],$ then $F(A)\subseteq A^{2};$ if $A\neq
A^{2},$ and $F(A)\nsubseteq A^{2},$ then there exists $x\in
F(A),x\notin A^{2}$ and a subalgebra $B$ of $A$ such that
$A^{2}\subseteq B, x\notin B$ and $\mathrm{dim}B=\mathrm{dim}A-1.$
Hence $B$ is a maximal subalgebras of $A$ which does not contain
$x.$ This contradicts  $x\in F(A).$ Therefore,  $F(A)\subseteq
A^{2}.$
\end{proof}

\begin{lemma} (see {\rm \cite{cm}})\label{l2.10}\,  Let $f$ be an endomorphism of a finite-dimensional vector superspace $V$ over $\mathbb{F}$
and let $\chi$ be a polynomial such that $\chi(f)=0.$ Then the
following statements hold:

$\mathrm{(1)}$ If $\chi =q_{1}q_{2}$ and $q_{1}, q_{2}$ are
relatively prime, then $V$ decomposes into a direct sum of
$f$-invariant subspaces $V=U\oplus W$ such that
$q_{1}(f)(U)=0=q_{2}(f)(W).$

$\mathrm{(2)}$ $V$ decomposes into a direct sum of $f$-invariant
subspaces $V =V_{0}\oplus V_{1},$ for which $f|_{V_{0}}$ is
nilpotent and $f|_{V_{1}}$ is invertible.
\end{lemma}

\begin{remark} \label{r2.11}
Note that, in the case where $V$ is finite dimensional, we may
choose $\chi$ to be the characteristic polynomial of $f.$ The
decomposition (2) is called the Fitting decomposition with respect
to $f.$  $V_{0},V_{1}$ are referred to as the Fitting -$0$ and
Fitting -$1$ components of $V,$ respectively.
\end{remark}

 An $n$-Lie superalgebra $A$ satisfies condition $\ast$ if the only subalgebra $K$ of $A$ with the property
$K+A^{2}=A$ is $K=A,$ where $A^{2}=[A,A,\cdots,A];$ an $n$-Lie
superalgebra satisfies condition $\ast\ast$ if $a_{i}\in
A_{0}(D(a_{1},\cdots,a_{n-1}))$ for some $1\leq i\leq n-1$ for
arbitrary $a_{i}\in A,$ where $A_{0}(D(a_{1},\cdots,a_{n-1}))=\{x\in
A| D^{r}(a_{1},\cdots,a_{n-1})(x)=0$ for some $r\}.$
\begin{theorem} \label{t2.151} Let $A$ be an $n$-Lie superalgebra over $\mathbb{F}.$  Then the following statements holds:

$\mathrm{(i)}$ If $A$ satisfies condition $\ast\ast$ and any maximal
subalgebra $M$ of $A$ is a weak ideal of $A,$ then $A$ is nilpotent.

$\mathrm{(ii)}$ If $A$ is nilpotent, then every maximal subalgebra $M$ of
$A$ is an ideal of $A.$
\end{theorem}
\begin{proof} $\mathrm{(i)}$ Assume that $A$ is not nilpotent. Then there exists a non-nilpotent left multiplication operator
$D(a_{1},\cdots,a_{n-1}).$ Put $D(a):=D(a_{1},\cdots,a_{n-1}).$
Since $D(a)$ is non-nilpotent, the Fitting-0 component
$A_{0}(D(a))\neq A.$ Let $M$ be a maximal subalgebra of $A$
containing $A_{0}(D(a)).$ Then  $a_{i}\in A_{0}(D(a))\subseteq M$
for some $1\leq i\leq n-1$ by assumption. Since the maximal
subalgebra $M$ of $A$ is a weak ideal of $A,$ $D(a)(A)\subseteq M.$
Since $D(a)$ is an automorphism on the Fitting-1 component
$A_{1}(D(a)),$ we obtain that $A_{1}=D(a)(A_{1})=A_{1}\cap M.$ Hence
$A_{1}\subseteq M.$ Then $A=A_{0}\oplus A_{1}\subseteq M\neq A.$
This is a contradiction. Thus all left multiplication operators are
nilpotent. Therefore, by Theorem $\ref{t1.1},$ $A$ is nilpotent.

$\mathrm{(ii)}$ We assume $A$ is nilpotent and $M$ is any maximal subalgebra of
$A.$ Then $R$ also acts nilpotently on $A$ for all $R\in D(A),$ where $D(A)$ is the vector space generated by all left multiplications of $A.$ Thus $R$ acts nilpotently on $A/M$ for all $R\in D(A).$ Then there is a $v\neq 0\in A/M$ such that $R(v)=0$ for all $R\in D(A).$ This means $R(v)\in M$ and hence $v\in N_{A}(M),$ where $N_{A}(M)=\{x\in A|[x,M,A,\cdots,A]\in M\},$ but since $v\neq 0\in A/M,$ we have that $v$ is not in $M,$ hence $M\subset N_{A}(M).$  By the maximality
of $M,$ then $N_{A}(M)=A,$
 i.e, $M$ is an ideal of $A.$
\end{proof}
\begin{corollary} \label{c2.151} Let $A$ be an $n$-Lie algebra over $\mathbb{F}.$  Then $A$
is nilpotent if and only if every maximal subalgebra $M$ of $A$ is
a weak ideal of $A.$
\end{corollary}
\begin{remark} \label{r2.151}  An $n$-Lie superalgebra
with condition $\ast\ast$ does exist. For example, Let $(A,\alpha)$
be an $n$-Lie superalgebra with basis $\{b,c\},$
$A=A_{\bar{0}}\oplus A_{\bar{1}},$ $A_{\bar{0}}=\mathbb{F}c,
A_{\bar{1}}=\mathbb{F}b,$ $\alpha=\bar{0},$ and its multiplication is as follow:
$[b,\cdots, b, c]=0, [b,\cdots, b]=c,$ then $b, c\in A_{0}(D(b,
\cdots, b, c)).$
\end{remark}

\begin{definition} \label{l2.9} An ideal $I$ of $n$-Lie superalgebra $A$ is called the Jacobson radical, if $I$ is the intersection of all maximal ideals of
$A,$ denoted by $J(A).$
\end{definition}

\begin{proposition} \label{t2.1511} For any $n$-Lie superalgebra
$A,$ $J(A)\subseteq A^{2}.$
\end{proposition}
\begin{proof} The proof is similar to that of Lemma $\ref{l2.91}.$
\end{proof}

\begin{definition} \label{l2.90} The ideal $I$ of $n$-Lie superalgebra $A$ is called $k$-solvable $(2\leq k\leq n)$ if $I^{(r)}=0$ for some $r\geq 0,$ where $I^{(0)}=I,$
$I^{(s+1)}=[\underbrace{I^{(s)},I^{(s)},\cdots,I^{(s)}}_{k},\underbrace{A,\cdots,A}_{n-k}]$
for some $s\geq 0.$  When $A=I,$ $A$ is called $k$-solvable $n$-Lie
superalgebra. Clearly, if $A$ is nilpotent, then it is
$k$-solvable$(k\geq 2)$.
\end{definition}

\begin{lemma} \label{t2.152} Let $A$ be a $k$-solvable $n$-Lie
superalgebra$(k\geq 2),$  then $J(A)=A^{(1)}.$
\end{lemma}
\begin{proof} By Proposition $\ref{t2.1511},$ $J(A)\subseteq A^{(1)}.$  We merely need verify $A^{(1)}\subseteq J(A).$  Let $I$ be an ideal of $A.$
As $A$ is $k$-solvable, $A/I$ is $k$-solvable and does not contain any
proper ideal of $A/I,$ hence
$[A/I,\cdots,A/I]=0,$ thus $A^{(1)}\subseteq I,$ by the
definition of the Jacobson radical, we have $A^{(1)}\subseteq J(A).$
Then we get $J(A)=A^{(1)}.$
\end{proof}
\begin{theorem} \label{l2.800} Let $A$ be a nilpotent $n$-Lie superalgebra over $\mathbb{F}.$  Then $F(A)=A^{(1)}=\phi(A)=J(A).$
\end{theorem}
\begin{proof}
 Since $A$ is nilpotent, by Theorem $\ref{t2.151}$ (ii),
any maximal subalgebra $T$ is an ideal of $A,$ $A/T$ is nilpotent
$n$-Lie superalgebra, and $A/T$ has no proper ideal, thus $[A/T,
\cdots, A/T]=0,$ $A^{(1)}\subseteq T,$ and $A^{(1)}\subseteq F(A).$
By Lemma $\ref{l2.91},$ $F(A)=A^{(1)}.$ Since $A$ is nilpotent, $A$ is
$k$-solvable, by Lemma $\ref{t2.152},$ $J(A)=A^{(1)}.$ Therefore,
$F(A)=\phi(A)=J(A)=A^{(1)}.$ The proof is complete.
\end{proof}

\begin{theorem} \label{t2.130} Let $A$ be an $n$-Lie superalgebra over
$\mathbb{F}.$ Then the following statements hold:

$\mathrm{(1)}$ If $A$ satisfies conditions $\ast\ast$ and $\ast,$
then $A$ is nilpotent.

$\mathrm{(2)}$ If $A$ is nilpotent, then the condition $\ast$ holds
in $A$.
\end{theorem}
\begin{proof} $\mathrm{(1)}$ Suppose that the
 condition $\ast$ holds in $A.$  Let $M$ be any maximal
subalgebra of $A.$ Since $M+A^{2}\neq A,$ $A^{2}\subseteq M,$ and $M$
is an ideal in $A.$ It follows from Theorem $\ref{t2.151}$ (i) $A$ is
nilpotent.

$\mathrm{(2)}$ Suppose that $A$ is nilpotent. By Theorem $\ref{l2.800},$ we have $A^{2}=F(A).$
Then $K+A^{2}=K+F(A)=A$ implies $K=A$ by Proposition $\ref{p2.3}.$
\end{proof}

\begin{corollary} \label{c2.152} Let $A$ be an $n$-Lie algebra over $\mathbb{F}.$  Then $A$
is nilpotent if and only if the
 condition $\ast$ holds in $A.$
\end{corollary}
\begin{definition} \label{d2.3}
 A subalgebra $T$ of an $n$-Lie superalgebra $A$ is called
 subinvariant if there exist subalgebras $T_{i}$ such that $A=T_{0}\supset T_{1}\supset T_{2}\supset \cdots \supset T_{n-1}\supset T_{n}=T$
where $T_{i}$ is an ideal in $T_{i-1}$ for $i=1,2,\cdots,n.$
\end{definition}
An upper chain, $C_{k},$ of length $k$ consists of subalgebras
$U_{0},U_{1},\cdots,U_{k}$ in $A$ such that $U_{0}=A$ and each
$U_{i}$ is maximal in $U_{i-1}$ for $i=1,2,\cdots,k.$ The
subinvariance number of $C_{k},$ $s(C_{k}),$ is defined to be the
number of $U_{i}\neq U_{0}=A$ which are subinvariant in $A;$ The
invariance number of $C_{k},$ $v(C_{k}),$ is defined as $k-s(C_{k})$
if $s(C_{k})\neq 0,$ and as $k$ otherwise. Then the invariance
number of $A,$ $v(A),$ is the maximum of $v(C_{k})$ for all $C_{k}$
of $A.$
\begin{lemma}\label{l2.100}\,  Let $A$ be a nonzero $n$-Lie superalgebra and $V$ be a maximal subalgebra of $A.$ If $V$ is not an ideal in $A,$
then $v(A)>v(V).$
\end{lemma}
\begin{proof} Suppose $C_{n}:$ $V=V_{0}\supset V_{1}\supset V_{2}\supset \cdots \supset
V_{n}$ is an upper chain of length $n$ in $V.$ Then $A\supset
V=V_{0}\supset V_{1}\supset V_{2}\supset \cdots \supset V_{n}$ is an
upper chain $C_{n+1}$ of length $n+1$ in $V.$ If $V_{i}(1\leq i\leq
n)$ is subinvariant in $A,$ then we have $$A=U_{0}\supset
U_{1}\supset U_{2}\supset \cdots \supset U_{k}=V_{i},$$ where
$U_{i}$ is an ideal in $U_{i-1}$ for $i=1,2,\cdots,k.$ We also have
$$V=A\cap V=U_{0}\cap V\supseteq U_{1}\cap V\supseteq \cdots \supseteq U_{k}\cap V=V_{i}.$$
Since $U_{i}$ is an ideal in $U_{i-1},$ $U_{i}\cap V$ an ideal in
$U_{i-1}\cap V$ and $V_{i}$ is subinvariant in $V.$ Hence, if
$V_{i}(1\leq i\leq n)$ is subinvariant in $A,$ then it is
subinvariant in $V.$ Since $V$ is not an ideal in $A,$
$s(C_{n+1})\leq s(C_{n}).$ If $s(C_{n+1})>0,$ then
$v(C_{n+1})=(n+1)-s(C_{n+1})\geq
(n+1)-s(C_{n})>n-s(C_{n})=v(C_{n}).$ If $s(C_{n+1})=0,$ then
$v(C_{n+1})=n+1>n\geq v(C_{n}).$
 Hence, $v(A)>v(V).$
\end{proof}
\begin{theorem}  \label{t2.13} Let $A$ be an $n$-Lie superalgebra over
$\mathbb{F}.$ Then the following statements hold:

$\mathrm{(1)}$ If $A$ satisfies condition $\ast\ast$ and $v(A)=v(U)$
for every proper subalgebra $U$ in $A,$ then $A$ is nilpotent.

$\mathrm{(2)}$ If $A$ is nilpotent, then for every proper subalgebra
$U$ in $A,$ $v(A)=v(U).$
\end{theorem}
\begin{proof} $\mathrm{(1)}$ Suppose $\mathrm{dim}(A)=n.$ Let $V$ be any maximal
subalgebra of $A$ such that $v(V)=v(A).$ Then by Lemma
$\ref{l2.100},$ $V$ is an ideal in $A.$  It follows from Theorem
$\ref{t2.151}$ (i) $A$ is nilpotent.

$\mathrm{(2)}$ If $A$ is nilpotent, then each subalgebra of $A$ is
subinvariant. Hence $v(A)=1.$ Since each subalgebra of $A$ is also
nilpotent, $v(V)=1,$ hence $v(A)=v(V).$
\end{proof}

\begin{corollary} \label{c2.152} Let $A$ be an $n$-Lie algebra over $\mathbb{F}.$  Then $A$
is nilpotent if and only if and $v(A)=v(U)$ for every proper
subalgebra $U$ in $A.$
\end{corollary}

\begin{theorem} \label{t2.150} Let $U$ be a subinvariant subalgebra of $n$-Lie superalgebra $A$ and $K$ an ideal of $U$ such that $K\subseteq F(A).$ If $U/K$ is
nilpotent, then $U$ is nilpotent.
\end{theorem}
\begin{proof} We have a chain of subalgebras $U=U_{r}\lhd U_{r-1}\lhd \cdots\lhd U_{0}=A.$ Let $a_{i}\in U(1\leq i\leq
n-1)$ and $D=D(a_{1},\cdots,a_{n-1}).$
Then $DU_{i-1}\subseteq U_{i}$ since $U_{i-1}\lhd U_{i}.$  Hence
$D^{r}A\subseteq U.$ But $U/K$ is nilpotent, so $D^{s}U\subseteq K$
for some $s.$  Thus, if $\mathrm{dim}(A)=t,$ we have
$D^{t}A\subseteq K.$  But $A=\mathrm{Im}(D^{t})\oplus
\mathrm{Ker}(D^{t}),$ so $A=K+E_{A}(D),$ where $E_{A}(D)=\{x\in
A|D^{r}(x)=0$ for some $r\}.$ But $K\subseteq F(A),$ so this implies
that $E_{A}(D)=A.$ Thus every $D(a)$ for all $a_{i}(1\leq i\leq
n-1)\in U$ is nilpotent and $U$ is nilpotent by Theorem
$\ref{t1.1}.$
\end{proof}

\begin{example} Let $(A,\alpha)$ be an $n$-Lie superalgebra with basis $\{b,c\},$
$A=A_{\bar{0}}\oplus A_{\bar{1}},$ $A_{\bar{0}}=\mathbb{F}c,
A_{\bar{1}}=\mathbb{F}b,$ $\alpha=\bar{0},$ and its multiplication is  
as follow: $[b,\cdots, b, c]=0, [b,\cdots, b]=c,$ then $A$ is
nilpotent, however $\mathrm{dim}(A/A^{2})=1.$
\end{example}
The above example shows the definition of $S^{\ast}$ algebra for an
$n$-Lie superalgebra is analogous to the case of Leibniz algebra,
thus we give the following definition:
\begin{definition} \label{d2.3}
An $n$-Lie superalgebra $A$ is called an $S^{\ast}$ algebra if every
proper non-abelian subalgebra $H$ of $A$ has either
$\mathrm{dim}(H/H^{2})\geq 2$ or
 $H$ is nilpotent and generated by one element.
\end{definition}

\begin{lemma}\label{l2.101} Let $A$ be a non-abelian nilpotent $n$-Lie superalgebra. Then we have either $\mathrm{dim}(A/A^{2})\geq 2$ or
 $A$ is generated by one element.
\end{lemma}
\begin{proof} Since $A$ is nilpotent, by Theorem $\ref{l2.800},$ one gets $A^{2}=F(A).$ It is clear that $\mathrm{dim}(A/A^{2})\neq 0$ since $A$ is nilpotent. If $\mathrm{dim}(A/A^{2})=1,$
then $A$ is generated by one element. Otherwise
$\mathrm{dim}(A/A^{2})\geq 2.$
\end{proof}

\begin{lemma}\label{l2.11}\,  Let $A$ be a non-nilpotent $n$-Lie superalgebra. If all proper subalgebras of $A$ are nilpotent, then $\mathrm{dim}(A/A^{2})\leq 1.$
\end{lemma}
\begin{proof} Suppose that $\mathrm{dim}(A/A^{2})\geq 2.$ Then there exist distinct maximal subalgebras $M$ and $N$ which contain $A^{2}.$
Hence $M$ and $N$ are nilpotent ideals, $A=M+N$ is nilpotent, a
contradiction.
\end{proof}

\begin{theorem} \label{t2.13} An $n$-Lie superalgebra $A$ is an $S^{\ast}$ algebra if and
only if $A$ is nilpotent.
\end{theorem}
\begin{proof} If $A$ is nilpotent, then every subalgebra of $A$ is nilpotent, so $A$ is an $S^{\ast}$ algebra by
Lemma $\ref{l2.101}.$ Conversely, suppose that there exists an
$S^{\ast}$ algebra that is not nilpotent. Let $A$ be the smallest
dimensional one that is not nilpotent. All proper subalgebras of $A$
are $S^{\ast}$ algebras, hence are nilpotent. Thus
$\mathrm{dim}(A/A^{2})\leq 1$ by Lemma $\ref{l2.11}.$ Since $A$ is
an $S^{\ast}$ algebra, it is generated by one element and is
nilpotent,  a contradiction.
\end{proof}

\begin{theorem} \label{t2.911} Let $(A,\alpha)$ be an $n$-Lie superalgebra and $D$ a derivation of $A.$ For $x_{1}, \cdots, x_{n}\in A,$ then  $D^{k}[x_{1},\cdots,x_{n}]=\sum\limits_{i_{1}+\cdots+i_{n}=k}a^{(k)}_{i_{1},\cdots,i_{n}}[D^{i_{1}}(x_{1}),\cdots,D^{i_{n}}(x_{n})],$ where $a^{(k)}_{i_{1},\cdots,i_{n}}\in \mathbb{F}.$
\end{theorem}
\begin{proof} We have induction on $k.$ If $k=1,$ then \begin{eqnarray*}&&D[x_{1},x_{2},\cdots,x_{n}]\\
&&=(-1)^{p(D)\alpha}[D(x_{1}),x_{2},\cdots,x_{n}]+(-1)^{p(D)(p(x_{1})+\alpha)}[x_{1},D(x_{2}),x_{3},\cdots,x_{n}]\\
&&\quad\quad+\cdots+
(-1)^{p(D)(p(x_{1})+\cdots+p(x_{n})+\alpha)}[x_{1},x_{2},\cdots,x_{n-1},D(x_{n})]\end{eqnarray*}
and the base case is satisfied. We now assume that the result holds
for $k$ and consider $k+1.$ Then
\begin{eqnarray*}&& D^{k+1}[x_{1},\cdots,x_{n}]\\
&&= D(\sum\limits_{i_{1}+\cdots+i_{n}=k}a^{(k)}_{i_{1},\cdots,i_{n}}[D^{i_{1}}(x_{1}),\cdots,D^{i_{n}}(x_{n})])\\
&&= \sum\limits_{i_{1}+\cdots+i_{n}=k}a^{(k)}_{i_{1},\cdots,i_{n}}\{(-1)^{p(D)\alpha}[D^{i_{1}+1}(x_{1}),\cdots,D^{i_{n}}(x_{n})]\\
&&\quad\quad+\cdots+(-1)^{p(D)\{p(x_{1})+
\cdots+p(x_{n})+\alpha+(i_{1}+\cdots+i_{n-1})p(D)\}}[D^{i_{1}}(x_{1}),\cdots,D^{i_{n}+1}(x_{n})]\}\\
&&=\sum\limits_{j_{1}+\cdots+j_{n}=k+1}a^{(k+1)}_{j_{1},\cdots,j_{n}}[D^{j_{1}}(x_{1}),\cdots,D^{j_{n}}(x_{n})].
\end{eqnarray*}
The last equality holds is because suppose that the array $(j_{1},\cdots,j_{n})$ satisfies $j_{1}+\cdots+j_{n}=k+1,$ then there must exist array $(i_{1},\cdots,i_{n})$ such that $i_{1}+\cdots+i_{n}=k$ and $m\in \{1,\cdots,n\}$ satisfies $i_{1}=j_{1}, \cdots, i_{m-1}=j_{m-1},i_{m}+1=j_{m},i_{m+1}=j_{m+1},\cdots, i_{n}=j_{n},$ that is, $(i_{1},\cdots,i_{m-1},i_{m}+1,i_{m+1},\cdots,i_{n})=(j_{1},\cdots,j_{m-1},j_{m},j_{m+1},\cdots,j_{n}).$
This proves the theorem.
\end{proof}

\begin{theorem} \label{t2.15} Let $A$ be an $n$-Lie superalgebra over
$\mathbb{F}.$ Suppose that $B$ is an ideal of $A$ and $C$ is an
ideal of $B$ such that $C\subseteq B\cap F(A).$ If $B/C$ is
nilpotent, then $B$ is nilpotent.
\end{theorem}
\begin{proof} Take any element $x_{i}(1\leq i\leq n-1)$ of $B.$ By Remark $\ref{r2.11},$ $A=A_{0}+A_{1}$ is the Fitting decomposition relative to $D(x),$
where $D(x)=D(x_{1},\cdots,x_{n-1})$ is nilpotent in $A_{0}$ and
$D(x)$ is an isomorphism of $A_{1}.$ So $A_{1}\subset B.$ Since
$B/C$ is nilpotent, there exists an integer $n$ such that
$A_{1}=D^{n}(x)(A_{1})\subset C.$ Then
$A=A_{0}+F(A).$ If $A_{0}$ is a subalgebra of $A,$ by Proposition $\ref{p2.3},$ it implies that
$A=A_{0}.$ Hence, $D(x)$ is nilpotent for any element $x_{i}(1\leq
i\leq n-1)\in B.$ Therefore, $B$ is nilpotent by virtue of Theorem
$\ref{t1.1}.$

It remains to show that $A_{0}$ is a subalgebra of $A.$  For $x_{1}, \cdots, x_{n}\in A,$ then by Theorem $\ref{t2.911},$ we have $$D(x)^{k}[x_{1},\cdots,x_{n}]=\sum\limits_{i_{1}+\cdots+i_{n}=k}a^{(k)}_{i_{1},\cdots,i_{n}}[D(x)^{i_{1}}(x_{1}),\cdots,D(x)^{i_{n}}(x_{n})].$$
If $x_{1}, \cdots, x_{n}\in A_{0},$ then $D(x)^{k}[x_{1},\cdots,x_{n}]=0,$ for a sufficiently big integer $k,$ hence $[x_{1},\cdots,x_{n}]\in A_{0}.$
\end{proof}
\begin{corollary} \label{p2.150} Let $A$ be an $n$-Lie superalgebra with $B\lhd A$ such that $B\subseteq F(A),$ then $B$ is nilpotent. In
particular, $\phi(A)$ is a nilpotent ideal of $A.$
\end{corollary}

\begin{definition} \label{l2.900} A nilpotent $n$-Lie superalgebra $A$ is said to be of class $t,$ if $A^{t+1}=0$ and $A^{t}\neq
0.$ We also denote by $cl(A)=t.$
\end{definition}

Put $AN^{i}=[A,\cdots, A, N^{i}]$ and $A^{j}N^{i}=[A,\cdots, A,
A^{j-1}N^{i}]$ for some $j>1.$

\begin{lemma} \label{l2.911} Let $A$ be an $n$-Lie superalgebra with $N\lhd A$ and $A/N^{2}$ be nilpotent. If $A^{m+1}\subset N^{2}$ for some minimal $m,$ then
$A^{u}N^{r}\subset N^{r+1}$ for $r>0$ where $u=(r-1)(n-1)(m-1)+m.$
\end{lemma}
\begin{proof} We have induction on $r.$ If $r=1,$ then $A^{(1-1)(n-1)(m-1)+m}N^{1}=A^{m}N\subseteq A^{m+1}\subset
N^{2}$ and the base case is satisfied. We now assume that the result
holds for $r$ and consider $r+1.$

Let $s=r(n-1)(m-1)+m$ and $u=(r-1)(n-1)(m-1)+m.$ By Theorem $\ref{t2.911},$ we obtain
$$A^{s}N^{r+1}=A^{s}[N^{r},N,A,\cdots,A]=\sum_{s_{1}+\cdots+s_{n}=s}[A^{s_{1}}N^{r}, A^{s_{2}}N, A^{s_{3}}A,\cdots, A^{s_{n}}A].$$
Suppose $s_{1}\geq u.$ Then by the induction hypothesis,
$A^{s_{1}}N^{r}\subset N^{r+1}$ and
$$\sum_{s_{1}+\cdots+s_{n}=s}[A^{s_{1}}N^{r}, A^{s_{2}}N, A^{s_{3}}A,\cdots, A^{s_{n}}A]\subset [N^{r+1},N,A,\cdots,A]\subset N^{r+2}.$$

Suppose $s_{1}<u.$ We claim there exists $s_{k}\geq m.$ Assume to
the contrary that $s_{j}<m$ for all $j.$ We obtain
$s=(s_{1})+(s_{2}+\cdots+s_{n})<u+(n-1)(m-1)=(r-1)(n-1)(m-1)+m+(n-1)(m-1)=r(n-1)(m-1)+m=s.$
But this is impossible. Hence there exists $s_{k}\geq m$ for some
$k.$ As a result $A^{s_{k}}N\subset N^{2}$ and using the
Filippov-Jacobi identity and skew super-symmetry, we obtain
\begin{eqnarray*}&& [A^{s_{1}}N^{r}, A^{s_{2}}N, A^{s_{3}}A,\cdots, A^{s_{k}}A,\cdots, A^{s_{n}}A]\\
&&= [N^{r},N,A,\cdots,A,N^{2},A,\cdots,A]\\
&&= [N^{r},N,A,\cdots,A,A,\cdots,A,N^{2}]\\
&&=[N^{r},N,A,\cdots,A,A,\cdots,A,[N,\cdots,N]]\\
&&=[[N^{r},N,A,\cdots,A,N,],N,\cdots,N]+[N,[N^{r},N,A,\cdots,A,N,],N,\cdots,N]\\
&&\quad+\cdots+[N,\cdots,N,[N^{r},N,A,\cdots,A,N,]]\\
&&\subseteq [N^{r+1},N,N,\cdots,N]\\
&&\subseteq [N^{r+1},N,A,\cdots,A]\\
&&=N^{r+2}.
\end{eqnarray*} This proves the lemma.
\end{proof}

\begin{theorem} \label{l2.801} Let $A$ be an $n$-Lie superalgebra with $N\lhd A.$  If $N^{t+1}=0$ and $(A/N^{2})^{m+1}=0,$ then
$cl(A)\leq tm+\frac{1}{2}t(t-1)(m-1)(n-1).$
\end{theorem}
\begin{proof}
 Using Lemma $\ref{l2.911},$ we observe that $A^{m+1}\subset N^{2}, A^{m+(n-1)(m-1)}N^{2}\subset N^{3},\cdots,$ $A^{m+(t-1)(n-1)(m-1)}N^{t}\subset N^{t+1}=0.$
Adding the exponents on the left-hand side, we see that
$A^{\omega}=0,$ where $\omega=tm+\frac{1}{2}t(t-1)(m-1)(n-1)+1,$ The
proof is complete.
\end{proof}

\begin{definition} \label{l2.901} Let $A$ be a nonzero
$n$-Lie superalgebra and $S$ be a subset of $A$ such that
$S\supseteq \{0\}.$ The normal closure of $S$ in $A,$ $S^{A},$ is
the smallest ideal in $A$ containing $S.$
\end{definition}

\begin{theorem} \label{t2.801} Let $A$ be a nonzero
$n$-Lie superalgebra over $\mathbb{F}.$  Then

$\mathrm{(i)}$  If $A$ satisfies condition $\ast\ast,$ then there
exists a nonzero nilpotent subalgebra $N$ in $A$ such that
$N^{A}=A.$

$\mathrm{(ii)}$ $A$ is nilpotent if and only if the subalgebra $N$
in $\mathrm{(i)}$ is $A.$
\end{theorem}
\begin{proof}
 $\mathrm{(i)}$ If $A$ is nilpotent, then we may take $N=A$ and $N^{A}=A^{A}=A.$ Consider the case that $A$ is not nilpotent. We use induction on the dimension of $A.$
A non-nilpotent $n$-Lie superalgebra of lowest dimension is two
dimensional, namely, $A=A_{\bar{0}}\oplus
A_{\bar{1}}(A_{\bar{0}}=\mathbb{F}x, A_{\bar{1}}=\mathbb{F}y)$ with
a bilinear skew super-symmetric bracket multiplication $[x,x,y]=y$
defined on $A.$ The normal closure of the one dimensional subalgebra
$\mathbb{F}x$ is $L.$ Assume that the theorem holds for all
non-nilpotent $n$-Lie superalgebras which dimension is less than
$n.$ Consider the case that $A$ is an $n$-dimensional non-nilpotent
$n$-Lie superalgebra. Then by Theorem $\ref{t2.151}$ (i), there
exists a maximal subalgebra $M$ in $A$ such that $M$ is not an ideal
in $A.$ Since the dimension of $M$ is less than $n,$ by our
inductive hypothesis there exists a nilpotent subalgebra $N$ in $M$
such that $N^{M}=M.$ We claim that $N^{A}\supseteq M.$ Since $N^{A}$
is an ideal in $A,$ $[A,\cdots, A, N^{A}]\subseteq N^{A}.$ In
particular, $[M,\cdots, M, N^{A}]\subseteq N^{A}.$ Since $M$ is a
subalgebra, $[M,\cdots, M, N^{A}\cap M]\subseteq N^{A}\cap M$ and
$N^{A}\cap M$ is an ideal in $M$ containing $N.$ Since $N^{M}$ is
the smallest ideal in $M$ containing $N,$ we have $N^{A}\cap
M\supseteq N^{M},$ i.e., we have $N^{A}\supseteq N^{A}\cap
M\supseteq N^{M}=M.$ Since $M$ is not an ideal of $A$ and $N^{A}$ is
an ideal of $A,$ $N^{A}\supset M.$ $N^{A}=A$ follow from $M$ being a
maximal subalgebra in $A.$

$\mathrm{(ii)}$ If $A=N$ and $N$ is nilpotent, $A$ is nilpotent.
Conversely, suppose $\{0\}\neq N\neq A.$ Then either $N$ is a
maximal subalgebra of nilpotent $n$-Lie superalgebra $A$ or $N$ is
contained in a maximal subalgebra $M$ of $A.$ By Theorem
$\ref{t2.151}$ (ii), every maximal subalgebra in $A$ is an ideal,
$N^{A}\subseteq M\neq A.$ That is a contradiction. Hence $N=A.$
 The proof is complete.
\end{proof}

\end{document}